\documentclass[reqno,b5paper]{amsart}
\usepackage{amssymb}
\usepackage{dsfont}
\usepackage{color}
\usepackage{csquotes}

\newtheorem{theorem}{Theorem}[section]

\newtheorem{corollary}[theorem]{Corollary}
\newtheorem{lemma}[theorem]{Lemma}
\newtheorem{example}[theorem]{Example}
\newtheorem{remark}[theorem]{Remark}
\newtheorem{definition}[theorem]{Definition}
\newtheorem{fact}[theorem]{Fact}

\newcommand{\T}{\mathbb{T}}
\newcommand{\Z}{\mathbb{Z}}
\newcommand{\N}{\mathbb{N}}

\def\cont{\mathfrak c}

\catcode`\@=12
\def\T{{\mathbb T}}

\def\Z{{\mathbb Z}}
\def\N{{\mathbb N}}
\def\R{{\mathbb R}}
\def\Q{{\mathbb Q}}

\def\cont{\mathfrak c}

\makeindex

\begin{document}
\title[Characterizing torsion subgroups of the circle]{Characterizing infinite torsion subgroups of the circle through arithmetic-type sequences}
\subjclass[2010]{Primary: 11J71, Secondary: 11K36, 20F38, 22B05}
\keywords{Circle group, torsion subgroup, characterized subgroup, topologically torsion, arithmetic sequence, arithmetic-type sequence}

\author{Ayan Ghosh}
\address{Department of Mathematics, Jadavpur University, Kolkata-700032, India}
\email {ayanghosh.jumath@gmail.com}

\author{Pratulananda Das}
\address{Department of Mathematics, Jadavpur University, Kolkata-700032, India}
\email {pratulananda@yahoo.co.in}

\begin{abstract}
A subgroup $H$ of the circle group $\T$ is called characterized by a sequence of integers $(u_n)$ if $H=\{x\in\T: \lim\limits_{n\to\infty} u_nx=0\}$. In a recent work [Das et al., Bull. Sci. Math. 199 (2025), 103580], the structure of characterized subgroups corresponding to arithmetic-type sequences was investigated. Building upon this work, we further show that a characterized subgroup associated with an arithmetic-type sequence is countable if and only if it is torsion. Further we prove that any infinite torsion subgroup of $\T$ can be characterized by an arithmetic-type sequence with bounded ratio. Moreover, our findings demonstrate that the dichotomy observed in Eggleston’s theorem [Theorem 16, Eggleston, Proc. Lond. Math. Soc. 54(2) (1952), 42--93] for arithmetic sequences does not extend, in general, to the broader class of arithmetic-type sequences.
\end{abstract}
\maketitle

\section{Introduction and background.\vspace{.3cm} \\}
An element $x$ of the circle group $\mathbb{T}$ is called topologically torsion corresponding to a sequence of integers $(u_n)$ if $\lim_{n\to\infty} u_nx=0_{\mathbb{T}}$. The most significant results concerning these elements were obtained for arithmetic sequences. In this note, we focus on arithmetic-type sequences (recently introduced in \cite{DGT1}) and conduct a thorough investigation of these elements and corresponding characterized subgroups, including their cardinality aspects. Since arithmetic sequences are a special case of arithmetic-type sequences, our findings offer a broader perspective on several intriguing properties associated with them. For notations and terminology, see Section 2.

Recall that an element $x$ of an abelian group $X$ is torsion if there exists an integer $k>0$ such that $kx = 0$ (more specifically called $k$-torsion in this case). An element $x$ of an abelian topological group $X$ is  \cite{B}:
\begin{itemize}
\item[(i)] {\em topologically torsion} if $n!x \rightarrow 0;$
\item[(ii)] {\em topologically $p$-torsion}, for a prime $p$, if $p^nx \rightarrow 0.$
\end{itemize}

It is obvious that any $p$-torsion element is topologically $p$-torsion. Armacost \cite{A} defined the subgroups
$$ X_p = \{x \in X: p^nx \rightarrow 0\} \ ~\mbox{and} ~ \ X! =  \{x \in X: n!x \rightarrow 0\}$$
of an abelian topological group $X$, and started to describe the elements of these subgroups. Note that the above two notions are just special cases of the following general notion considered in (Section 4.4.2, \cite{DPS}).

\begin{definition}
Let $(a_n)$ be a sequence of integers, the subgroup
$$
t_{(a_n)}(\T) := \{x\in \T: a_nx \to 0\mbox{ in } \T\}
$$
of $\T$ is called {\em a characterized} $($by $(a_n))$ {\em subgroup} of $\T$.
\end{definition}

The theory of so-called characterized subgroups has a rich and distinguished history, with some of its foundational motivations rooted in the study of the distribution of sequences of real multiples modulo one. The classical case, where such sequences are uniformly distributed modulo one, was first examined in the seminal work of Hermann Weyl \cite{W} (see also \cite{BDBW, BDMW1, DPS} for further developments and recent contributions in this area). The interest to the case when these sequences are small, actually  null sequences, stems from Harmonic Analysis where $A$-sets (short for Arbault sets) were introduced in \cite{A1}, in particular, trigonometric series (see also \cite{BuKR, El}).

Again coming back to Armacost,

(a) obviously $t_{(p^n)}(\T)$ contains
the Pr\" ufer group $\Z(p^\infty)$. Armacost \cite{A} proved that $t_{(p^n)}(\T)$ simply coincides with $\Z(p^\infty)$ and $x$ is a topologically $p$-torsion element if and only if $supp(x)$ (defined after Fact \ref{lemmanew}) is finite.

(b) at the same time, he \cite{A} posed the problem to describe the group $\T! = t_{(n!)}(\T)$ which was much later resolved independently and almost simultaneously in \cite[Chap. 4]{DPS} and by J.-P. Borel \cite{B2}.

In particular, in both the above mentioned instances, the sequences of integers, concerned are arithmetic sequences. Recall that a sequence of positive integers $(a_n)$ is called an arithmetic sequence if
$$
1<a_1<a_2<a_3< \ldots <a_n< \ldots \ ~ \ \mbox{ and } a_n|a_{n+1} \ ~ \ \mbox{ for every } n\in\N  .
$$
For an arithmetic sequence of integers $(a_n)$, the sequence of ratios $(q_n)$ is defined as
$$
q_1=a_1 \mbox{     \ and \      } q_n=\frac{a_n}{a_{n-1}}  \mbox{ for  } n \geq 2.
$$
Note that, $A\subseteq\N$ is called
\begin{itemize}
\item[(i)] $q$-bounded if the sequence of ratios $(q_n)_{n\in A}$ is bounded.
\item[(ii)] $q$-divergent if the sequence of ratios $(q_n)_{n\in A}$  diverges to $\infty$.
\end{itemize}
We say $(a_n)$ is $q$-bounded ($q$-divergent) if $\N$ is $q$-bounded ($q$-divergent).

As previously discussed, some of the most intriguing cases involve topologically $\underline{a}$-torsion elements characterized by arithmetic sequences. The foundational results of Armacost \cite{A} and Borel \cite{B2} were extended to a broader context involving arbitrary arithmetic sequences in \cite{DI1}, where a complete description of the topologically $\underline{a}$-torsion elements corresponding to such sequences was provided. In a notable development, however, a non-arithmetic sequence $(\zeta_n)$ was introduced in \cite{DK}, defined as follows:
\begin{equation}\label{eqnonarith}
1,2,4,6,12, 18, 24,  \ldots, n!, 2\cdot n!, 3 \cdot n!, \ldots , n \cdot n!, (n+1)!, \ldots
\end{equation}
It was established in \cite{DK} that $t_{(\zeta_n)}(\T) = \Q/\Z$. Motivated by this observation, for an arithmetic sequence $(a_n)$ the following general class of non-arithmetic sequences was introduced in \cite{DG8}. Let $(d_n^{a_n})$ be an increasing sequence of integers formed by the elements of the set,
\begin{equation}\label{nonarithdef}
\{ra_{k-1} \ : \ 1\leq r< q_{k}\}.
\end{equation}
When there is no confusion regarding the sequence $(a_n)$, we simply denote this sequence by $(d_n)$. Note that for $a_n=n!$ corresponding non-arithmetic sequence $(d_n)$ coincides with the sequence $(\zeta_n)$. A through investigation regarding the subgroup $t_{(d_n)}(\T)$, in particular cardinality related questions and its relation with $t_{(a_n)}(\T)$ was carried out in \cite{DG8}. The whole investigation reiterates that these characterized subgroups are infinitely generated countable torsion subgroup of the circle group $\T$.

An increasing sequence of integers $(e_n)$ is called an arithmetic-type sequence if
$$
(a_n)\subseteq (e_n^{a_n}) \subseteq (d_n^{a_n}).
$$
Again, when there is no confusion regarding the sequence $(a_n)$, we simply denote this sequence by $(e_n)$. Therefore, for any arithmetic sequence $(a_n)$, we always have
\begin{equation}\label{eqincluchar}
 t_{(d_n)}(\T) \subseteq t_{(e_n)}(\T) \subseteq t_{(a_n)}(\T).
\end{equation}

\section{Notations.\vspace{.3cm} \\}
Throughout $\R$, $\Q$, $\Z$, $\T$ and $\N$ will stand for the set of all real numbers, the set of all rational numbers, the set of all integers, the set of all complex numbers on the unit circle and the set of all natural numbers \em(note that we do not consider zero as a natural number) \em respectively. The first three are equipped with their usual abelian group structure and the circle group $\T=\R/\Z$ is equipped with the operation of addition mod 1. Following \cite{Ka}, we may identify $\T$ with the interval [0,1] identifying 0 and 1. Any real valued function $f$ defined on $\T$ can be identified with a periodic function defined on the whole real line $\R$ with period 1, i.e., $f(x+1)=f(x)$ for every real $x$. When referring to a set $X\subseteq \T$ we assume that $X\subseteq [0,1]$ and $0\in X$ if and only if $1\in X$. For a real $x$, we denote its fractional part by $\{x\}$ and $\|x\|$ the distance from the integers, i.e., $\min\big\{\{x\},1-\{x\}\big\}$.\\

Let $G$ be an abelian group. The cyclic subgroup of $G$ generated by $g\in G$ is denoted by $\langle g\rangle$. The cyclic group of order $n$ is denoted by $\Z(n)$. If $S\subseteq G$ then the subgroup generated by $S$ is defined as
$$
\langle S\rangle = \bigg\{\sum\limits_{i=1}^j m_is_i : j\in\N, m_i\in\Z, s_i\in S\bigg\}.
$$
The subgroup of torsion elements ($p$-torsion) of $G$ is denoted by $t(G)$ (resp., by $t_p(G)$). For the circle group $\T$ we denote $t_p(\T)$ by $\Z(p^\infty)$. A subgroup $H$ of $G$ is called torsion if $H=t(H)$.

For arithmetic sequences, the following facts will be used in this sequel time and again. So, before moving onto our main results here we recapitulate that once.
\begin{fact}\label{lemmanew}\cite{DI1}
For any arithmetic sequence $(a_n)$ and $x\in\T$, we can find a unique sequence $c_n\in [0,q_n-1]$ such that
\begin{equation}\label{canonical:repr}
x=\sum\limits_{n=1}^{\infty}\frac{c_n}{a_n},
\end{equation}
where $c_n<q_n-1$ for infinitely many $n$.
\end{fact}
%

For $x\in\T$ with canonical representation (\ref{canonical:repr}), we define
\begin{itemize}
\item[$\bullet$] $supp(x) = \{n\in \N: c_n \neq 0\}$,
\item[$\bullet$] $supp^q(x)=\{n\in\N\ : \ c_n=q_n-1\}$.
\end{itemize}
For an arithmetic sequence $(a_n)$, note that for each $j\in\N$,
\begin{equation}\label{eqsum}
\sum\limits_{i=j}^\infty \frac{c_i}{a_i} \leq \sum\limits_{i=j}^\infty \frac{q_i-1}{a_i} = \sum\limits_{i=j}^\infty \bigg(\frac{1}{a_{i-1}} - \frac{1}{a_i} \bigg) \leq \frac{1}{a_{j-1}}.
\end{equation}
\begin{lemma}\label{uconlmain}\cite[Lemma 3.1]{DR}
For $x\in\T$ with canonical representation (\ref{canonical:repr}), for every natural $n > 1$ and every non-negative integer $t$,
\begin{equation}\label{eqlemain1}
\{a_{n-1}x\}=\frac{c_n}{q_n}+\frac{c_{n+1}}{q_nq_{n+1}}+\ldots+\frac{c_{n+t}}{q_n\ldots q_{n+t}}+\frac{\{a_{n+t}x\}}{q_n\ldots q_{n+t}}.
\end{equation}
In particular, for $t=1$, we get
\begin{equation}\label{eqlemain2}
\{a_{n-1}x\}=\frac{c_n}{q_n}+\frac{c_{n+1}}{q_nq_{n+1}}+\frac{\{a_{n+1}x\}}{q_nq_{n+1}}.
\end{equation}
\end{lemma}
The following facts are well known for integer norm:
\begin{eqnarray}\label{eqinorm1}
&\bullet& \ \mbox{for any } \ A\subseteq\N\ \mbox{ and }\ y_n=x_n+z_n \ \mbox{ if } \ \lim\limits_{n\in A} z_n=0, \ \mbox{ then } \  \lim\limits_{n\in A}\|y_n\|=\lim\limits_{n\in A}\|x_n\|. \\
&\bullet& \ \mbox{if } (x_n)  \mbox{ is  bounded  with } x_n\geq 0\mbox{ for all }n\in A \mbox{ and } \lim_{n\in A} \|y_n\|=0, \mbox{ then } \lim_{n\in A}\|x_ny_n\|=0.\\
&\bullet& \ \mbox{for any integer } \ n\ \mbox{ and } \ \mbox{for any real number } x, \|n+x\|=\|x\|.
\end{eqnarray}
Also note that for any $r\in\N$,
\begin{equation}\label{eqr}
\bullet \ \mbox{if } \ r\{a_nx\}<1\ \mbox{ then } \ \{ra_nx\}=r\{a_nx\}.
\end{equation}
\begin{equation}\label{eqrn}
\bullet \ \mbox{if } \ r\|a_nx\|<\frac{1}{2} \ \mbox{ then } \ \|ra_nx\|=r\|a_nx\|.
\end{equation}
\begin{remark}\label{uconr1}
Since $(a_n)$ is a subsequence of $(e_n)$, we can write $a_k=e_{n_k}$. More precisely, we can write $e_{(n_{k-1}+i-1)}=r^k_ia_{k-1}$ where $1\leq i \leq n_{k}-n_{k-1}$ and $r^k_i\in [1,q_{k}-1]$ with $r^k_1=1$. In particular, if $(e_n)=(d_n)$ then $n_{k}-n_{k-1}=q_{k}-1$ and $d_{(n_{k-1}+i-1)}=ia_{k-1}.$ We simply denote $r^k_i$ by $r(e)$ when there is no confusion in the indexing.
\end{remark}
\begin{remark}\label{uconr2}
Let $a_k=e_{n_k}$ and $A\subseteq \N$. We define
$$
L(A)=\bigcup\limits_{k\in A} [n_{k-1},n_{k}-1].
$$
Then, in view of Remark \ref{uconr1}, for any subsequence $(e_n)_{n\in B}$ of $(e_n)$ there exists a unique $A\subseteq\N$ such that $B\subseteq L(A)$ and if $A_1\subsetneq A$ then $B\not\subseteq L(A_1)$.
\end{remark}

\section{Cardinality related observations. \vspace{.3cm} \\}
In this section, we explore the connection between the torsion subgroups of the circle group $\T$ and the characterized subgroups associated with arithmetic-type sequences. Theorems \ref{thuncoun1} and \ref{thcountabletor} play a central role in establishing this relationship. Before presenting the proofs of these theorems, we recall a characterization of topologically torsion elements corresponding to arithmetic-type sequences, as established in \cite{DGT1}. This result will be repeatedly utilized throughout the article.

\begin{theorem}\label{main theorem DGT1}\cite[Theorem 2.5]{DGT1}
Suppose $x\in\T$ and $(e_n)$ is an arithmetic-type sequence. Then $x\in t_{(e_n)}(\mathbb{T})$ if and only if $supp(x)$ is finite or if $supp(x)$ is infinite and for each infinite $A\subseteq \mathbb{N}$ the following holds:
\begin{itemize}
\item[$(a)$] If $A$ is $q$-bounded then :
\begin{itemize}
\item[$(a1)$]  If $A\subset^*supp(x)$ then $(A+1)\subset^* supp(x)$, $A\subset^*supp^q(x)$ and $\lim\limits_{n\in A}$$\frac{c_{n+1}+1}{q_{n+1}}=1$.\\
Moreover, if $A+1$ is $q$-bounded then $(A+1)\subset^*supp^q(x)$.
\item[$(a2)$] If $A\cap supp(x)$ is finite then $\lim\limits_{n\in A}\frac{c_{n+1}}{q_{n+1}}=0$.\\
Moreover, if $A+1$ is $q$-bounded then $(A+1)\cap supp(x)$ is finite.
\end{itemize}
\item[$(b)$]  If $A$ is $q$-divergent then :
\begin{itemize}
\item[$(b1)$] If $(A+1)\subset^* supp(x)$ and
\begin{itemize}
\item[$(i)$] if $A+1$ is $q$-bounded then $\lim\limits_{n\in A}\|\frac{r_{n}(e)}{q_n}(c_{n}+1)\|=0$.
\item[$(ii)$] if $A+1$ is $q$-divergent then $\lim\limits_{n\in A}\|\frac{r_{n}(e)}{q_n}(c_{n}+\frac{c_{n+1}}{q_{n+1}})\|=0$.
\end{itemize}
\item[$(b2)$] If $(A+1)\cap supp(x)$ is finite then $\lim\limits_{n\in A}\|\frac{r_{n}(e)c_n}{q_n}\|=0$.	
\end{itemize}
\end{itemize}
\end{theorem}
Consider $x\in\T$ with $supp(x)=A$. We set $\mathcal{G}_A(\T)=\{y\in\T: supp(y)\subseteq A$ and $c_n(y)=c_n(x)$ for each $n\in supp(y)\}$. Our next result plays a prominent role in proving one of the main result of this section namely, Theorem \ref{thuncoun1}.
\begin{lemma}\label{leuncoun1}
Consider $x\in\T$ with $supp(x)=A$. If $A$ is an infinite subset of $\N$ then $|\mathcal{G}_A(\T)|=\mathfrak{c}$.
\end{lemma}
\begin{proof}
Let us write $A=\{l_1<l_2<\ldots<l_k<\ldots\}$. We fix a sequence $(\delta_i)\in \{0,1\}^\N$ and define $B^\delta = \bigcup\limits_{k=1}^{\infty}l_{2k+ \delta_k}$ (it is to be noted that $B^\delta $ is actually obtained from $A$ by taking at each stage $k$ either $l_{2k}$ or $l_{2k+1}$ depending on the choice imposed by $(\delta_i)$). Clearly $B^\delta \ne B^\gamma$ for distinct $\delta, \gamma \in \{0,1\}^\N$. Consequently, this provides an injective map given by
$$
\{0,1\}^\N \ni \delta \to B^\delta.
$$
Now, for each $x_\delta\in\T$ with $supp (x_\delta)=B^\delta$ and $c_n(x_\delta)=c_n(x)$ for each $n\in supp(x_\delta)$, we have $x_\delta\in G_A(\T)$. Since $|\{0,1\}^\N| = \mathfrak c$, it readily follows that $|\mathcal{G}_A(\T)|=\mathfrak{c}$.
\end{proof}
The following result demonstrates that if the characterized subgroup associated with an arithmetic-type sequence contains even a single non-torsion element, then it necessarily contains uncountably many such elements.
\begin{theorem}\label{thuncoun1}
Suppose $x\in t_{(e_n)}(\T)$ and $supp(x)$ is infinite. Then $|t_{(e_n)}(\T)|=\mathfrak{c}$.
\end{theorem}
\begin{proof}
Suppose $x\in t_{(e_n)}(\T)$ and $supp(x)$ is an infinite non-cofinite set. Let us write
$$
 B=supp(x)\setminus (supp(x)-1).
$$
Since $supp(x)$ is not cofinite, we must have $B$ is infinite. Also observe that $B\subseteq supp(x)$ and $(B+1)\cap supp(x)=\emptyset$. If possible assume that there exists an infinite set $B_1\subseteq B$ such that $B_1$ is $q$-bounded. Then item $(a1)$ of \cite[Theorem 2.5]{DGT1} entails that $(B_1+1)\subseteq^* supp(x)$ which contradicts the fact that $(B+1)\cap supp(x)=\emptyset$. Indeed we conclude that $B$ is $q$-divergent. As a consequence, item $(b2)$ of \cite[Theorem 2.5]{DGT1} holds true. Therefore, for each $r(e)\in[1,q_{k}-1]$ we have
\begin{equation}\label{eqnoncofin2}
\lim\limits_{n\in B}\|\frac{r(e)c_{n}(x)}{q_n}\|=0.
\end{equation}
In particular, taking $r(e)=1$ we have
$$
\lim\limits_{n\in B}\|\frac{c_{n}(x)}{q_n}\|=0.
$$
Now consider $y\in\T$ such that $supp(y)\subseteq B$ and for each $n\in supp(y)$,
$$
c_n(y) = \begin{cases}
   c_n(x) & \text{for}\ c_n(x)\leq \frac{q_{n}}{2} \\
   q_n- c_n(x) & \text{otherwise}.
  \end{cases}
$$
Then it is evident that $supp(y)\cap (supp(y)+1)=\emptyset$ and
$$
\lim\limits_{n\in supp(y)} \frac{c_n(y)}{q_n} =0.
$$
Note that for each $n\in\N$, we have
\begin{equation}\label{eqnorm}
\|\frac{r(e)(q_n - c_{n}(x))}{q_n}\|=\|\frac{r(e)c_{n}(x)}{q_n}\|.
\end{equation}
Therefore, in view of Eq(\ref{eqnorm}), Eq(\ref{eqnoncofin2}) ensures that
$$
\lim\limits_{n\in supp(y)}\|\frac{r(e)c_{n}(y)}{q_n}\|=\lim\limits_{n\in supp(y)\subseteq B}\|\frac{r(e)c_{n}(x)}{q_n}\|=0.
$$
Since $B$ is $q$-divergent and $supp(y)\subseteq B$, \cite[Corollary 3.2]{DGT1} ensures that $y\in t_{(e_n)}(\T)$. Therefore, for each $A\subseteq B$, we can construct an element $y_A\in\T$ such that $supp(y_A)=A$ with $c_n(y_A)=c_n(y_B)$ for each $n\in supp(y_A)$ and $y_A\in t_{(e_n)}(\T)$. Indeed we conclude that $\mathcal{G}_B(\T)\subseteq t_{(e_n)}(\T)$. In view of Lemma \ref{leuncoun1}, it is evident that $|t_{(e_n)}(\T)|=\mathfrak{c}$.

So, let us consider $supp(x)$ is cofinite. We set
$$
B_0'=supp(x)\setminus supp^q(x).
$$
Since $supp^q(x)$ is not cofinite (see Remark \ref{lemmanew}), we must have $B_0'$ is infinite. Also observe that $B_0'\subseteq supp(x)$ and $B_0' \cap supp^q(x)=\emptyset$. If possible assume that there exists an infinite set $B_1'\subseteq B_0'$ such that $B_1'$ is $q$-bounded. Then item $(a1)$ of \cite[Theorem 2.5]{DGT1} entails that $B_1'\subseteq^* supp^q(x)$ which contradicts the fact that $B_0'\cap supp^q(x)=\emptyset$. Indeed we conclude that $B_0'$ is $q$-divergent. Let us write
$$
B_0'=\{n_k: k\in\N\} \ \mbox{ and set } \ B'=\{n_{2k}:k\in\N\}.
$$
Then $B'$ is again $q$-divergent and $B'\cap (B'+1) =\emptyset$. Since $supp(x)$ is cofinite, we also have $B'+1\subseteq^* supp(x)$. Consequently, item $(b1)$ of \cite[Theorem 2.5]{DGT1} holds true, i.e., if $B'+1$ is $q$-bounded then for each $r(e)\in[1,q_{k}-1]$ we have
\begin{equation}\label{eq2ai2}
\lim\limits_{n\in B'}\|\frac{r(e)}{q_n}(c_{n}(x)+1)\|=0 ,
\end{equation}
and, if $B'+1$ is $q$-divergent then for each $r(e)\in[1,q_{k}-1]$ we have
\begin{equation}\label{eq2aii2}
 \lim\limits_{n\in B'}\|\frac{r(e)}{q_n}(c_{n}(x)+\frac{c_{n+1}(x)}{q_{n+1}})\|=0 \mbox{ and } \lim\limits_{n\in B'}\|\frac{c_{n+1}(x)}{q_{n+1}}\|=0.
\end{equation}
Then Eq(\ref{eq2ai2}) and Eq(\ref{eq2aii2}) entails that there exists $\epsilon_n\in [0,1]$ such that
\begin{equation}\label{eqepsilon}
\lim\limits_{n\in B'}\|\frac{r(e)}{q_n}(c_{n}(x)+\epsilon_n)\|=0 \mbox{ and } \lim\limits_{n\in B'}\|\epsilon_n\|=0.
\end{equation}
In particular, taking $r(e)=1$ in Eq(\ref{eqepsilon}), we have
$$
\lim\limits_{n\in B'}\|\frac{c_{n}(x)}{q_n}\|=\lim\limits_{n\in B'}\|\frac{c_{n}(x)+1}{q_n}\|=0.
$$
Since $B'\cap supp^q(x)=\emptyset$, for each $n\in B'$ we have $c_n(x) \leq q_n-2$. Now consider $y\in\T$ such that $supp(y)\subseteq B'$ and for each $n\in supp(y)$
$$
c_n(y) = \begin{cases}
   q_n-c_n(x) -1 & \text{for}\ c_n(x)\geq \frac{q_{n}}{2} \text{and}\ \epsilon_n\geq \frac{1}{2} \\
   q_n-c_n(x)    & \text{for}\ c_n(x)\geq \frac{q_{n}}{2} \text{and}\ \epsilon_n\leq \frac{1}{2} \\
   c_n(x)        & \text{for}\ c_n(x)\leq \frac{q_{n}}{2} \text{and}\ \epsilon_n\leq \frac{1}{2} \\
   c_n(x)+1      & \text{for}\ c_n(x)\leq \frac{q_{n}}{2} \text{and}\ \epsilon_n\geq \frac{1}{2}.
  \end{cases}
$$
Then it is evident that $supp(y)\cap (supp(y)+1)=\emptyset$ and
$$
\lim\limits_{n\in supp(y)} \frac{c_n(y)}{q_n}=0.
$$
Indeed, in view of Eq(\ref{eqnorm}), Eq(\ref{eqepsilon}) ensures that
$$
\lim\limits_{n\in supp(y)}\|\frac{r(e)c_{n}(y)}{q_n}\|=0.
$$
Therefore, from \cite[Corollary 3.2]{DGT1}, we conclude that $y\in t_{(e_n)}(\T)$. Then proceeding exactly same way as the non-cofinite case, in view of Lemma \ref{leuncoun1}, we conclude that $|t_{(e_n)}(\T)|=\mathfrak{c}$.
\end{proof}
In 1983, J.-P. Borel \cite{B1} demonstrated that every countable subgroup of the circle group $\T$ can be characterized by some increasing sequence of integers, a result that serves as a counterpart to Eggleston’s theorem \cite[Theorem 16]{E}. This naturally leads to the question of whether a smaller subclass of sequences (strictly contained within the class of all increasing sequences of integers) might suffice to characterize all countable subgroups of $\T$. It was shown in \cite{BDBW} that the subgroup $\Q/\Z$ cannot be characterized by any arithmetic sequence. Given that the class of arithmetic-type sequences properly contains the class of arithmetic sequences, and that there exist arithmetic-type sequences which do characterize $\Q/\Z$, it is natural to inquire whether Borel’s result extends to this broader class. Our previous theorem, together with the following lemma, provides a negative answer to this question. In particular, Example \ref{excounter} demonstrates the existence of countable subgroups of $\T$ that cannot be characterized by any arithmetic-type sequence.
\begin{corollary}\label{corotorcountable}
For any arithmetic-type sequence $(e_n)$, $t_{(e_n)}(\T)$ is torsion if and only if it is countable.
\end{corollary}
\begin{proof}
First we assume that $t_{(e_n)}(\T)$ is torsion. Then $t_{(e_n)}(\T)$ is countable follows from the fact that $t_{(e_n)}(\T)\subseteq t(\T)=\Q/\Z$. Now, on the contrary, assume that $t_{(e_n)}(\T)$ is not torsion. Then there exists $x\in t_{(e_n)}(\T)$ such that $supp(x)$ is infinite. Then Theorem \ref{thuncoun1} ensures that $t_{(e_n)}(\T)$ is uncountable.
\end{proof}
\begin{example}\label{excounter}
Consider any irrational number $\gamma\in (0,1)$. Then Corollary \ref{corotorcountable} ensures that the cyclic subgroup of $\T$ generated by $\gamma$ (i.e., $\langle\gamma + \Z \rangle$) cannot be characterized by arithmetic-type sequences.
\end{example}
\begin{corollary}\label{coronontor}
For any arithmetic-type sequence $(e_n)$, the followings are equivalent:
\begin{itemize}
\item[$(i)$] $t_{(e_n)}(\T)\not= t_{(d_n)}(\T)$.
\item[$(ii)$] $t_{(e_n)}(\T)$ contains non-torsion elements.
\item[$(iii)$] $t_{(e_n)}(\T)$ is uncountable.
\item[$(iv)$] $|t_{(e_n)}(\T)|=\mathfrak{c}$
\item[$(v)$] $|t_{(e_n)}(\T)\setminus t_{(d_n)}(\T)|=\mathfrak{c}$.
\end{itemize}
\end{corollary}
\begin{proof}
$(i) \Rightarrow (ii)$ follows from \cite[Theorem 2.7]{DG8}. $(iii)$ and $(iv)$ are equivalent. $(ii) \Rightarrow (iv)$ follows from Theorem \ref{thuncoun1}. $(iv) \Rightarrow (v)$ follows from \cite[(iii) Corollary 2.4]{DG8}. $(v) \Rightarrow (i)$ is trivial.
\end{proof}
Corollary \ref{corotorcountable} implies that if the characterized subgroup corresponding to an arithmetic-type sequence is countable, then it must be torsion. This naturally raises the question of whether the converse holds namely, whether every countable torsion subgroup of $\T$ can be characterized by some arithmetic-type sequence. The following result affirms this and provides a positive answer to the converse statement.
\begin{theorem}\label{thcountabletor}
Any infinite torsion subgroup of $\T$ can be characterized by an arithmetic-type sequence of integers with bounded ratio.
\end{theorem}
\begin{proof}
Let $G$ be an infinite torsion subgroup of $\T$. Therefore, we need to show that there exists an arithmetic-type sequence $(e_n)$ with bounded ratio (i.e., the sequence $(\frac{e_{n+1}}{e_n})$ is bounded) such that $G=t_{(e_n)}(\T)$. Note that any countable subgroup $G$ of $\T$ can be written as the countable increasing union of finitely generated abelian groups, i.e.,
$$
G=\bigcup\limits_{k=1}^\infty H_k, \mbox{ where each } H_k \mbox{ are finitely generated and } H_k\subsetneq H_{k+1}.
$$
Indeed, $H_k$ is torsion and finitely generated for each $k\in\N$. Therefore, we conclude that $H_k$ is a finite cyclic subgroup of $\T$ for each $k\in\N$. Let order of $H_k$ be $a_k$. Then observe that $a_k$ properly divides $a_{k+1}$ for each $k\in\N$, i.e., $(a_k)$ is an arithmetic sequence. Now, consider any arithmetic-type sequence $(e_n)$ corresponding to the arithmetic sequence $(a_n)$ with bounded ratio. Since the ratio sequence $(\frac{e_{n+1}}{e_n})$ is bounded, \cite[Theorem 16]{E} entails that $t_{(e_n)}(\T)$ is countable. Now, consider any $x\in t_{(e_n)}(\T)$. Then Theorem \ref{thuncoun1} ensures that $supp_{(a_n)}(x)$ is finite. Let us write
$$
x=\frac{c_{n_1}}{a_{n_1}}+\frac{c_{n_2}}{a_{n_2}}+ \ldots + \frac{c_{n_i}}{a_{n_i}}.
$$
Since $H_k$ is the unique cyclic subgroup of order $a_k$ for each $k\in\N$, it is easy to see that $x\in H_{n_i}\subseteq G$. Also, note that for each $k\in\N$, $a_k$ divides $e_l$ for each $l\geq n_k$. Therefore, it is clear that $G\subseteq t_{(e_n)}(\T)$. Thus, we conclude that $G=t_{(e_n)}(\T)$.
\end{proof}
\begin{corollary}
Every countable subgroups of $\T$ can be characterized by an arithmetic-type sequence if and only if it is torsion.
\end{corollary}
\begin{proof}
Let $G$ be a countable subgroup of $\T$. First we assume that $G$ is characterized by an arithmetic-type sequence. Then Corollary \ref{corotorcountable} ensures that $G$ is torsion. Conversely, we consider that $G$ is a torsion subgroup of $\T$. Then Theorem \ref{thcountabletor} ensues that $G$ can be characterized by an arithmetic-type sequence.
\end{proof}

\section{Eggleston's dichotomy vs characterized subgroup. \vspace{.3cm} \\}
Eggleston \cite[Theorem 16, Theorem 17]{E} observed that for any increasing sequence of integers $(u_n)$ the asymptotic behavior of the ratio sequence $(\frac{u_{n+1}}{u_n})$ has a strong impact on the size of $t_{(u_n)}(\T)$:
\begin{itemize}
\item[(E1)]  $t_{(u_n)}(\T)$ is countable if $(\frac{u_{n+1}}{u_n})$ is bounded,
\item[(E2)] $|t_{(u_n)}(\T)|=\cont$ if $\frac{u_{n+1}}{u_n} \to \infty$.
\end{itemize}
But the result fails to conclude anything if the ratio sequence $(\frac{u_{n+1}}{u_n})$ is any unbounded sequence. In 2012, D. Dikranjan and D. Impieri had shown that $t_{(a_n)}(\mathbb{T})$ is countable if and only if $(q_n)$ is bounded \cite[Corollary 3.8]{DI1}, i.e., for an arithmetic sequence, Eggleston's result becomes a dichotomy. So, it is natural to ask whether this dichotomy holds for any arithmetic-type sequence of integers. Our next result provides a negative solution in this direction. It also entails that the condition ``bounded ratio", imposed on the arithmetic sequence in Theorem \ref{thcountabletor}, is not necessary.
\begin{theorem}\label{Dicho}
For any arithmetic sequence $(a_n)$ with unbounded ratio, there exist an arithmetic-type sequence $(e_n)$ with unbounded ratio such that $t_{(e_n)}(\T)$ is countable.
\end{theorem}
\begin{proof}
Let $(a_n)$ be an arithmetic sequence with unbounded ratio. Therefore, there exists $D\subseteq \N$ such that $(q_n)_{n\in D}$ is divergent. We define the sequence $(e_n)$ by the elements of the following set
$$
\bigg\{ra_{k-1} \ : \ r\in \{1\} \cup \bigg[\bigg\lfloor\frac{q_k}{3}\bigg\rfloor+1,q_k-1\bigg]\bigg\}.
$$
Since $(a_n)$ is a subsequence of $(e_n)$ we write $e_{n_k}=a_k$ for each $k\in\N$. Now observe that for each $k\in D-1$, we have
\begin{equation*}
\frac{e_{n_k+1}}{e_{n_k}} = \frac{\bigg(\bigg\lfloor\frac{q_{k+1}}{3}\bigg\rfloor+1\bigg)a_k}{a_k} \geq \frac{q_{k+1}}{3}.
\end{equation*}
Indeed we have $(e_n)$ has unbounded ratio. So it is left to show that $t_{(e_n)}(\T)$ is countable.

First we will show that if $supp(x)$ is infinite then $x\notin t_{(e_n)}(\T)$. On the contrary let us assume that $supp(x)$ is infinite and $x\in t_{(e_n)}(\T)$. Now two cases can arise, i.e., either $supp(x)$ is cofinite or $supp(x)$ is not cofinite.
\begin{itemize}
\item[\underline{Case-1} :] Here we assume that $supp(x)$ is not cofinite. Let us write
$$
A=supp(x)\setminus (supp(x)-1).
$$
Since $supp(x)$ is not cofinite, we must have $A$ is infinite. Also observe that $A\subseteq supp(x)$ and $(A+1)\cap supp(x)=\emptyset$. If possible assume that there exists an infinite set $A_1\subseteq A$ such that $A_1$ is $q$-bounded. Then item $(a1)$ of \cite[Theorem 2.5]{DGT1} entails that $(A_1+1)\subseteq^* supp(x)$ which contradicts the fact that $(A+1)\cap supp(x)=\emptyset$. Indeed we conclude that $A$ is $q$-divergent. As a consequence, item $(b2)$ of \cite[Theorem 2.5]{DGT1} holds true. Therefore, for each $r(e)\in[1,q_{k}-1]$ we have
\begin{equation}\label{eqnoncofin}
\lim\limits_{n\in A}\|\frac{r(e)c_{n}}{q_n}\|=0.
\end{equation}
In particular, taking $r(e)=1$ we have
$$
\lim\limits_{n\in A}\|\frac{c_{n}}{q_n}\|=0, \mbox{ i.e.,}
$$
\begin{itemize}
    \item[(1a)] either there exists an infinite subset $A_1$ of $A$ such that $\lim\limits_{n\in B_1}\frac{c_{n}}{q_n}=0$,
    \item[(1b)] or there exists an infinite subset $A_1'$ of $A$ such that $\lim\limits_{n\in B_2}\frac{c_{n}}{q_n}=1$.
\end{itemize}
\begin{itemize}
\item[\underline{Case-1a} :] First we consider $\lim\limits_{n\in A_1}\frac{c_{n}}{q_n}=0$. Now, for $\epsilon\in (0,\frac{1}{8})$ there exists $n_\epsilon$ such that for each $n\in A_1\setminus [1,n_\epsilon]$, we have
$$
0<\frac{c_{n}}{q_{n}}<\epsilon \mbox{    (since $c_n\geq 1$)}.
$$
For each $n\in A_1\setminus [1,n_\epsilon]$, we set
\begin{equation}\label{eqA1t}
t=\bigg\lfloor \frac{q_{n}}{5c_{n}}\bigg\rfloor,\ \mbox{ i.e., }\ \frac{q_{n}}{5c_{n}}-1<t \leq \frac{q_{n}}{5c_{n}}
\end{equation}
and,
$$
r(e)=\begin{cases}
   t & \text{for}\ t\in\big[\big\lfloor\frac{q_{k}}{3}\big\rfloor+1,q_{k}-1\big]\\
   q_k-t & \text{for}\ t\notin\big[\big\lfloor\frac{q_{k}}{3}\big\rfloor+1,q_{k}-1\big].
  \end{cases}
$$
Then Eq (\ref{eqA1t}) entails that
$$
\frac{3}{40}<\frac{1}{5}-\frac{c_{n}}{q_{n}}<\frac{tc_{n}}{q_{n}} \leq \frac{1}{5}.
$$
Therefore, observe that for each $n\in A_1\setminus [1,n_\epsilon]\subseteq A$,
$$
\big\|\frac{r(e) c_n}{q_n}\big\|= \big\|\frac{(q_n - t)c_n}{q_n}\big\|=\big\|\frac{t c_n}{q_n}\big\|\geq \frac{3}{40}
$$
$-$ which is a contradiction (in view of Eq (\ref{eqnoncofin})).
\item[\underline{Case-1b} :] Let us consider $\lim\limits_{n\in A_1'}\frac{c_{n}}{q_n}=1$, i.e., $\lim\limits_{n\in A_1'}\frac{q_n - c_n}{q_n}=0$. Now, for $\epsilon\in (0,\frac{1}{8})$ there exists $n_\epsilon$ such that for each $n\in A_1'\setminus [1,n_\epsilon]$, we have
$$
0<\frac{q_n- c_{n}}{q_{n}}<\epsilon \mbox{    (since $1\leq c_n\leq q_n-1$)}.
$$
For each $n\in A_1'\setminus [1,n_\epsilon]$, we set
\begin{equation}\label{eqB1t}
t=\bigg\lfloor \frac{q_{n}}{5(q_n-c_{n})}\bigg\rfloor,\ \mbox{ i.e., }\ \frac{q_{n}}{5(q_n-c_{n})}-1<t \leq \frac{q_{n}}{5(q_n-c_{n})}
\end{equation}
and,
$$
r(e)=\begin{cases}
   t & \text{for}\ t\in\big[\big\lfloor\frac{q_{k}}{3}\big\rfloor+1,q_{k}-1\big]\\
   q_k-t & \text{for}\ t\notin\big[\big\lfloor\frac{q_{k}}{3}\big\rfloor+1,q_{k}-1\big].
  \end{cases}
$$
Then Eq (\ref{eqB1t}) entails that
$$
\frac{3}{40}<\frac{1}{5}-\frac{q_n-c_{n}}{q_{n}}<\frac{t(q_n-c_{n})}{q_{n}} \leq \frac{1}{5}.
$$
Therefore, observe that for each $n\in A_1'\setminus [1,n_\epsilon]\subseteq A$,
$$
\big\|\frac{r(e) c_n}{q_n}\big\|= \big\|\frac{r(e)(q_n-c_{n})}{q_n}\big\|= \big\|\frac{(q_n - t)(q_n-c_{n})}{q_n}\big\|=\big\|\frac{t (q_n-c_{n})}{q_n}\big\|\geq \frac{3}{40}.
$$
$-$ which is a contradiction (in view of Eq (\ref{eqnoncofin})).
\end{itemize}
So, we conclude that Case-1 is not possible.
\item[\underline{Case-2} :] Now, let us assume that $supp(x)$ is cofinite. We write
$$
B=supp(x)\setminus supp^q(x).
$$
Since $supp^q(x)$ is not cofinite (see Remark \ref{lemmanew}), we must have $B$ is infinite. Also observe that $B\subseteq supp(x)$ and $B \cap supp^q(x)=\emptyset$. If possible assume that there exists an infinite set $B_1\subseteq B$ such that $B_1$ is $q$-bounded. Then item $(a1)$ of \cite[Theorem 2.5]{DGT1} entails that $B_1\subseteq^* supp^q(x)$ which contradicts the fact that $B\cap supp^q(x)=\emptyset$. Indeed we conclude that $B$ is $q$-divergent. Since $supp(x)$ is cofinite, we also have $B+1\subseteq^* supp(x)$. Consequently, item $(b1)$ of \cite[Theorem 2.5]{DGT1} holds true, i.e., if $B+1$ is $q$-bounded then for each $r(e)\in[1,q_{k}-1]$ we have
\begin{equation}\label{eq2ai}
\lim\limits_{n\in B}\|\frac{r(e)}{q_n}(c_{n}+1)\|=0 ,
\end{equation}
and, if $B+1$ is $q$-divergent then for each $r(e)\in[1,q_{k}-1]$ we have
\begin{equation}\label{eq2aii}
 \lim\limits_{n\in B}\|\frac{r(e)}{q_n}(c_{n}+\frac{c_{n+1}}{q_{n+1}})\|=0 .
\end{equation}
In particular, taking $r(e)=1$ in item $(b1)$ of \cite[Theorem 2.5]{DGT1}, we have
$$
\lim\limits_{n\in B}\|\frac{c_{n}}{q_n}\|=\lim\limits_{n\in B}\|\frac{c_{n}+1}{q_n}\|=0 , \mbox{ i.e.,}
$$
\begin{itemize}
    \item[(2a)] either there exists an infinite subset $B_1$ of $B$ such that $\lim\limits_{n\in B_1}\frac{c_{n}+1}{q_n}=0$,
    \item[(2b)] or there exists an infinite subset $B_1'$ of $B$ such that $\lim\limits_{n\in B_1'}\frac{c_{n}+1}{q_n}=1$.
\end{itemize}
\begin{itemize}
\item[\underline{Case-2a} :] First we consider $\lim\limits_{n\in B_1}\frac{c_{n}+1}{q_n}=0$. For $\epsilon\in (0,\frac{1}{8})$ there exists $n_\epsilon$ such that for each $n\in B_1\setminus [1,n_\epsilon]$, we have
$$
0<\frac{c_{n}}{q_{n}}<\frac{c_{n}+1}{q_{n}}<\epsilon \mbox{    (since } c_n\geq 1).
$$
Let $B+1$ is $q$-bounded. For each $n\in B_1\setminus [1,n_\epsilon]$, we set $r(e)$ exactly as in Case-1a. Then Eq (\ref{eqA1t}) entails that
$$
\frac{3}{40}=\frac{1}{5}-\frac{1}{8}<\frac{c_n+1}{5c_n}-\frac{c_n+1}{q_{n}}<\frac{t(c_n+1)}{q_{n}} \leq \frac{c_n+1}{5c_n}\leq \frac{2}{5}.
$$
Now observe that for each $n\in B_1\setminus [1,n_\epsilon]\subseteq B$,
$$
\big\|\frac{r(e) (c_n+1)}{q_n}\big\|= \big\|\frac{(q_n - t)(c_n+1)}{q_n}\big\|=\big\|\frac{t (c_n+1)}{q_n}\big\|\geq \frac{3}{40}
$$
$-$ which is a contradiction (in view of Eq (\ref{eq2ai})).\\
So, let us consider $B+1$ is $q$-divergent. Therefore, we also have
\begin{equation}\label{eqnormequal}
\lim\limits_{n\in B}\|\frac{c_{n+1}}{q_{n+1}}\|=0 .
\end{equation}
For each $n\in B_1\setminus [1,n_\epsilon]$, we set $r(e)$ exactly as in Case-1a. Then Eq (\ref{eqA1t}) entails that
$$
\frac{3}{40}<\frac{1}{5}-\frac{c_{n}}{q_{n}}<\frac{t c_n}{q_{n}}\leq\frac{t(c_n+\frac{c_{n+1}}{q_{n+1}})}{q_{n}}\leq \frac{t(c_n+1)}{q_{n}} \leq \frac{c_n+1}{5c_n}\leq \frac{2}{5}.
$$
Now, in view of Eq (\ref{eqnormequal}), observe that for each $n\in B_1\setminus [1,n_\epsilon]\subseteq B$,
$$
\big\|\frac{r(e) (c_n+\frac{c_{n+1}}{q_{n+1}})}{q_n}\big\|= \big\|\frac{(q_n - t)(c_n+\frac{c_{n+1}}{q_{n+1}})}{q_n}\big\|=\big\|\frac{t (c_n+\frac{c_{n+1}}{q_{n+1}})}{q_n}\big\|\geq \frac{3}{40}
$$
$-$ which is a contradiction (in view of Eq (\ref{eq2aii})).
\end{itemize}
\item[\underline{Case-2b} :]  Let us consider $\lim\limits_{n\in B_1'}\frac{c_n}{q_n}=\lim\limits_{n\in B_1'}\frac{c_{n}+1}{q_n}=1$, i.e., $\lim\limits_{n\in B_1'}\frac{q_n - c_n-1}{q_n}=\lim\limits_{n\in B_1'}\frac{q_n - c_n}{q_n}=0$. For $\epsilon\in (0,\frac{1}{15})$ there exists $n_\epsilon$ such that for each $n\in B_1'\setminus [1,n_\epsilon]$, we have
$$
0<\frac{q_n- c_{n}-1}{q_{n}}<\frac{q_n- c_{n}}{q_{n}}<\epsilon \mbox{  (since } 1\leq c_n\leq q_n-2).
$$
Let $B+1$ is $q$-bounded. For each $n\in B_1'\setminus [1,n_\epsilon]$, we set $r(e)$ exactly as in Case-1b. Then Eq (\ref{eqB1t}) entails that
$$
\frac{t(q_n - c_n- 1)}{q_{n}} \leq \frac{q_n - c_n- 1}{5(q_n - c_n)}\leq \frac{1}{5}
$$
and,
\begin{equation}\label{eqnormgrt}
\frac{t(q_n - c_n- 1)}{q_{n}}>\frac{q_n - c_n- 1}{5(q_n - c_n)}- \frac{q_n - c_n- 1}{q_{n}}>\frac{1}{10}-\frac{1}{15}=\frac{1}{30}.
\end{equation}
Now observe that for each $n\in B_1'\setminus [1,n_\epsilon]\subseteq B$,
\begin{eqnarray*}
\big\|\frac{r(e) (c_n+1)}{q_n}\big\| = \big\|\frac{r(e) (q_n - c_n- 1)}{q_n}\big\| &=& \big\|\frac{(q_n - t)(q_n - c_n- 1)}{q_n}\big\| \\ &=&\big\|\frac{t (q_n - c_n- 1)}{q_n}\big\|\geq \frac{1}{30}
\end{eqnarray*}
$-$ which is a contradiction (in view of Eq (\ref{eq2ai})). \\
So, let us consider $B+1$ is $q$-divergent.  For each $n\in B_1'\setminus [1,n_\epsilon]$, we set $r(e)$ exactly as in Case-1b. Then Eq (\ref{eqB1t}) and Eq (\ref{eqnormgrt}) ensure that
$$
\frac{1}{30}<\frac{t(q_n - c_n-1)}{q_{n}}\leq \frac{t (q_n - c_n- \frac{c_{n+1}}{q_{n+1}})}{q_n} \leq \frac{t(q_n - c_n)}{q_{n}} \leq \frac{1}{5}.
$$
Now, in view of Eq (\ref{eqnormequal}), observe that for each $n\in B_1'\setminus [1,n_\epsilon]\subseteq B$,
\begin{eqnarray*}
\big\|\frac{r(e) (c_n+\frac{c_{n+1}}{q_{n+1}})}{q_n}\big\| &=& \big\|\frac{r(e) (q_n - c_n- \frac{c_{n+1}}{q_{n+1}})}{q_n}\big\| \\ &=& \big\|\frac{(q_n - t)(q_n - c_n- \frac{c_{n+1}}{q_{n+1}})}{q_n}\big\| \\ &=&\big\|\frac{t (q_n - c_n- \frac{c_{n+1}}{q_{n+1}})}{q_n}\big\|\geq \frac{1}{30}.
\end{eqnarray*}
$-$ which is a contradiction (in view of Eq (\ref{eq2aii})).
\end{itemize}
So, we conclude that Case-2 is not possible. Therefore, it is evident that if $supp(x)$ is infinite then $x\not\in t_{(e_n)}(\T)$. As a consequence, \cite[Theorem 2.3]{DG8} ensures that $t_{(e_n)}(\T)\subseteq  t_{(d_n)}(\T)$. Finally, in view of \cite[Corollary 2.4, (iii)]{DG8}, we conclude that $t_{(e_n)}(\T)$ is countable.
\end{proof}

%
%
\noindent{\textbf{Acknowledgement:}} The first author as RA and the second author as PI are thankful to SERB(DST) for the CRG project (No. CRG/2022/000264) and the second author is also thankful to SERB for the MATRICS project (No. MTR/2022/000111) during the tenure of which this work has been done.

\end{document}